\documentclass[12pt]{article}
\usepackage{amsfonts,mathrsfs,amscd,amssymb,amsthm,amsmath,bm,graphicx,psfrag,amsrefs}
\usepackage[ruled]{algorithm2e}
\usepackage[noend]{algpseudocode}
\usepackage{authblk,hyperref}

\large\normalsize

\newtheorem{thm}{Theorem}[section]
\newtheorem{prop}[thm]{Proposition}

\newtheorem{defn}[thm]{Definition}

\newcommand{\C}{\mathbb{C}}

\newcommand{\Z}{\mathbb{Z}}
\newcommand{\R}{\mathbb{R}}
\newcommand{\Q}{\mathbb{Q}}
\newcommand{\lcm}{\mathrm{lcm}}
\newcommand{\aff}{\mathrm{aff}}
\newcommand{\Col}{\operatorname*{Col}}
\newcommand{\vol}{\operatorname*{vol}}

\usepackage{hyperref}
\hypersetup{colorlinks,%
citecolor=red,%
linkcolor=blue,%
}
\usepackage{color}

\begin{document}

\title{Ehrhart quasi-polynomials of rational polytopes by real dilations\thanks{Research was supported by RGC Grant 16308821.}}

\author{Ying Cao and Beifang Chen\thanks{Department of Mathematics, Hong Kong University of Science and Technology, Clear Water Bay, Hong Kong, \href{mailto:ycaobf@connect.ust.hk}{ycaobf@connect.ust.hk}, \href{mailto:mabfchen@ust.hk}{mabfchen@ust.hk}}}


\date{}
\maketitle

\begin{abstract}
This paper is to study the Ehrhart function $L(P,t)$ of a rational $n$-polytope $P$, defined as the number of lattice points of dilated polytopes $tP$ with real numbers $t\geq 0$. It turns out that $L(P,t)$ is a quasi-polynomial of real variable $t$ in the sense that 
\[
L(P,t)=\sum_{k=0}^{n} c_k(P,t)t^k, \quad t\geq 0,
\]
where $c_k(P,t)$ are periodic piecewise polynomials of degree $n-k$ if $\aff P$ contains the origin, and are periodic functions vanishing almost everywhere otherwise. When $P$ is a rational simplex $\sigma$, the coefficient functions $c_k(\sigma,t)$ are given explicitly in terms of vertex information of the simplex $\sigma$. Moreover, the reciprocity law still holds.
\end{abstract}


\section{Introduction}

It is well-known that for a lattice polytope $P$ the number of lattice points of dilated polytopes $nP:=\{n\alpha:\alpha\in P\}$ by integers $n\geq 0$, denoted $L(P,n)$, is a polynomial function of degree $\dim P$ in terms of $n$, known as the {\em Ehrhart polynomial} of $P$, due to the work of Ehrhart~\cite{ehrhart1962polyedres,ehrhart1967probleme,ehrhart1967problemei}. Since Ehrhart's seminal work, Ehrhart theory has attracted significant attention \cite{barvinok1999algorithmic,cappell1994genera,brion1997lattice,diaz1997ehrhart,beck2007computing} and developed into a rich theory with diverse applications. 

For a rational polytope $P$, the counting function $L(P,n)$ is no longer in polynomial pattern but a quasi-polynomial. Recall that a {\em quasi-polynomial of degree $m$} is defined by Stanley \cite[p.\,210]{stanley2011enumerative} as a function $f:{\Bbb Z}\rightarrow{\Bbb C}$ (or $f:\Z_{\geq 0}\to\C$) such that 
\begin{align}\label{definition of quasi-polynomial}
f(n)=\sum_{k=0}^m c_k(n)n^k, 
\end{align}
where $c_m(n)\not\equiv 0$, $n^0\equiv 1$, and each $c_k:{\Bbb Z}\rightarrow{\Bbb C}$ (or $c_k:\Z_{\geq 0}\to\C$) is a periodic function with integer period. It is easy to see that the quasi-polynomial pattern \eqref{definition of quasi-polynomial} is unique by considering the restriction to each residue class modulo a common period of the coefficient periodic functions $c_k(n)$. See also Barvinok \cite{barvinok2006computing} for linear identities for quasi-polynomials, which implies the uniqueness of the quasi-polynomial pattern \eqref{definition of quasi-polynomial}.


\begin{thm}[Ehrhart~\cite{ehrhart1962polyedres} and McMullen~\cite{mcmullen1978lattice}]
For each $P$ rational polytope of ${\Bbb R}^N$, the counting function $L(P,n):=\#(nP\cap{\Bbb Z}^N)$ has the following pattern 
\begin{equation}\label{eq:Nonnegative-count}
L(P,n)
=\sum_{k=0}^{\dim P} c_k(P,n)n^k, \quad n\in\Z_{\geq 0},
\end{equation}
where each $c_k(P,\cdot):{\Bbb Z}_{\geq 0}\rightarrow{\Bbb Q}$ is a periodic function, with period $s_k$ as the smallest positive integer such that the affine space $s_k\aff F$ contains lattice points for all $k$-dimensional faces $F$ of $P$.
\end{thm}

Let $P$ be a rational polytope of dimension $d$ and denote by $\mathring{P}:=P\smallsetminus\partial P$ its relative interior. For negative integers $n$, let $L(P,n)$ be defined as 
\begin{align}\label{def negative value of integer dilation}
L(P,n):=\sum_{k=0}^{\dim P} c_k\left(P,n-\left\lfloor\frac{n}{s_k}\right\rfloor s_k\right)\cdot n^k,\quad n\in\Z_{<0}.
\end{align}
Then the function $L(P,n)$ of integer variable $n\in\Z$ is a quasi-polynomial, known as the {\em Ehrhart quasi-polynomial} of the rational polytope $P$.

The quasi-polynomial is defined either on $\Z_{\geq 0}$ or on $\Z$, and the two are clearly equivalent. Notice that the counting function $\#(n\mathring{P}\cap\Z^N)$ for integers $n\geq 0$ is not necessarily a quasi-polynomial since $\#(0\mathring{P}\cap\Z^N)=1$. However, the function $L(\mathring{P},n)$ defined on $\Z_{\geq 0}$, given by
\[
L(\mathring{P},n):=\begin{cases}
\#(n\mathring{P}\cap\Z^N), & n>0,\\
(-1)^{\dim P}, & n=0,
\end{cases}
\]
is indeed a quasi-polynomial, which is related to the quasi-polynomial $L(P,n)$ via the reciprocity law $L(\mathring{P},n)=(-1)^{\dim P}L(P,-n)$ for $n\in\Z$. We see that both $L(P,n)$ and $L(\mathring{P},n)$ need to be defined on all integers. For the sake of clear statement on the reciprocity law, it is best to define quasi-polynomials on $\Z$ rather than on $\Z_{\geq 0}$ at the very beginning.

Linke~\cite{linke2011rational}, Stapleton~\cite{stapledon2017counting} and Beck et al.~\cite{beck2021rational} studied rational Ehrhart quasi-polynomials of rational polytopes by rational dilations. Linke~\cite{linke2011rational} mentioned that rational dilations can be extended to real dilations by taking limits, which can be done in principle; while Beck et al.~\cite{beck2021rational} also mentioned the real Ehrhart function $L(P,t)$ for rational polytopes $P$ with codenominator $r$. Baldoni et al.~\cite{baldoni2013intermediate} provided an algorithm to compute the real Ehrhart quasi-polynomials in the form of explicit step-polynomials. 

In this paper, we study real dilations directly without bothering the extension from the case of rational dilations to real dilations. We delve into the simplices $\sigma$ decomposed from $P$ and their interiors $\mathring{\sigma}$ to obtain concrete patterns directly of the periodic function $c_k(\sigma,t)$ for real numbers $t\geq 0$, the periodic function $c_k(\mathring{\sigma},t)$ for real numbers $t>0$, and $c_k(P,t)$ by summing over the given decomposition.

\begin{defn}
A {\em real quasi-polynomial of degree $m$} is a function $f:\R\to\C$ such that
\begin{align}\label{real quasi-polynomial}
f(t)=\sum_{k=0}^m c_k(t)t^k,
\end{align}
where each $c_k:\R\to\C$ is a periodic function and adopt the convention $t^0\equiv 1$ for all $t\in\R$.
\end{defn} 


Let $P\subset\R^N$ be a rational polytope, and denote by $d:=d(P)$ the smallest positive real number such that $dP$ is an integral polytope, called  the {\em denominator} of $P$. The key objects are the counting functions
\begin{align}
L(P,t)&:=\#(tP\cap\Z^N),\quad t\ge 0,\\
L(\mathring{P},t)&:=\#(t\mathring{P}\cap\Z^N),\quad t>0.
\end{align}
The following theorem is the main result of this paper, it captures the quasi-polynomial pattern of $L(P,t)$ and $L(\mathring{P},t)$.

\begin{thm}  \label{main1}
Let $P\subset\R^N$ be a rational polytope with denominator $d$. Then the counting functions $L(P,t)$ and $L(\mathring{P},t)$ have the quasi-polynomial pattern:
\begin{align}
L(P,t)&=\sum_{k=0}^{\dim P} c_k(P, t)\, t^k,\quad t\ge 0 \label{real close quasi}\\
L(\mathring{P},t)&=\sum_{k=0}^{\dim P} c_k(\mathring{P}, t)\,t^k,\quad t>0 \label{real open quasi}
\end{align}
where $c_k(P, t)$ for $t\ge 0$ and $c_k(\mathring{P}, t)$ for $t>0$ are periodic functions of period $d$. And $L(P,t)$ and $L(\mathring{P},t)$ extend by \eqref{real close quasi} and \eqref{real open quasi} to quasi-polynomials of real variable $t$ as follows: 
\begin{itemize}
\item For $t<0$, where $c_k(P,t):=c_k(P,t+md)$ with integers $m$ such that $t+md\ge 0$;
\item For $t\le 0$, where $c_k(\mathring{P},t):=c_k(\mathring{P},t+md)$ with integers $m$ such that $t+md>0$.
\end{itemize}
\end{thm}

By definition a real quasi-polynomial $f(t)$ takes $t\in\R$, and the periodic coefficient functions $c_k(P,t)$ and $c_k(\mathring{P},t)$ (period $d$) are likewise defined for every real argument. Equations \eqref{real close quasi} and \eqref{real open quasi} therefore describe $L(P,t)$ and $L(\mathring{P},t)$ on all of $\R$; however, the lattice-point interpretation applies only when $t\ge0$ for $P$ and $t>0$ for $\mathring{P}$, while negative values merely evaluate the quasi-polynomials. For $L(\mathring{P},t)$ at $t=0$, the value $L(\mathring{P},0)$ is sometimes overlooked by some authors and it must be given by
\[
L(\mathring{P},0):=c_0(\mathring{P},0)=\sum_{k=0}^{\dim P}c_k(\mathring{P}, 0)0^k.
\]

Theorem \ref{main1} implies the case of integer dilations studied by Ehrhart~\cite{ehrhart1962polyedres} and McMullen~\cite{mcmullen1978lattice}, and the case of rational dilations studied by Linke~\cite{linke2011rational}. Furthermore, for a rational simplex $\sigma$, we introduce functions $\mathbf{s}_k(\sigma;\cdot),\bar{\mathbf{s}}_k(\sigma;\cdot):\R^{n+1}\to\R$ as follows 
\begin{align}
\mathbf{s}_k(\sigma;x_0,\dots,x_n)&=\#D(\sigma,-dx_0)\cdot s_k(x_1,\dots,x_n)\label{eq:multivariable functions for closed simplex}\\
\bar{\mathbf{s}}_k(\sigma;y_0,\dots,y_n)&=\#\bar{D}(\sigma,-dy_0)\cdot s_k(y_1,\dots,y_n)\label{eq:multivariable functions for open simplex}
\end{align}
where $s_{n-k}$ are elementary symmetric polynomial of $n$ variables, $D(\sigma,\ell)$ is the set of lattice points of $\ell\sigma\cap \sum_{i=0}^n[0,d\alpha_i)$, $\bar{D}(\sigma,\ell)$ is the set of lattice points of $\ell\mathring{\sigma}\cap \sum_{i=0}^n(0,d\alpha_i]$. 

With the help of $\mathbf{s}_\sigma$ and $\bar{\mathbf{s}}_\sigma$, we obtain concrete formulas for the coefficients $c_k(\sigma,t)$ and $c_k(\mathring{\sigma},t)$:
\begin{align}
c_k(\sigma,t)&=\frac{1}{n!d^k}\sum_{i=0}^{n}\mathbf{s}_{n-k}(\sigma;x_0(i,t),\dots,x_n(i,t))\label{concrete coefficient of closed simplex}\\
c_k(\mathring{\sigma},t)&=\frac{1}{n!d^k}\sum_{j=1}^{n+1}\bar{\mathbf{s}}_{n-k}(\sigma;y_0(j,t),\dots,y_n(j,t))\label{concrete coefficient of open simplex}
\end{align}
where the variable functions $x_l$ and $y_l$ ($0\leq l\leq n$) in \eqref{concrete coefficient of closed simplex} and \eqref{concrete coefficient of open simplex} are given by
\begin{align}
x_l(i,t)&=l-i-\left\{\mbox{$\frac{t}{d}$}\right\}, \quad 0\leq i\leq n,\ t\in\R,\label{eq:variable functions for closed simplex}\\
y_l(j,t)&=l-j-\left\langle\mbox{$\frac{t}{d}$}\right\rangle, \quad 1\leq j\leq n+1,\ t\in\R.\label{eq:variable functions for open simplex}
\end{align}
The real Ehrhart quasi-polynomials $L(P,t)$ and $L(\mathring{P},t)$ of a rational polytope $P$ can be computed from these formulas by decomposing $P$ into rational (relatively open) simplices. 


The formulas \eqref{real close quasi} and \eqref{real open quasi} for counting the number of lattice points in real dilations of rational polytopes preserve the nice reciprocity law of Ehrhart (quasi-)polynomials in integer dilations as follows.



\begin{thm}[Reciprocity law for rational simplices] \label{reciprocity law for Ehrhart function}
For $\sigma$ each rational simplex of $\R^N$ with denominator $d$, the real Ehrhart quasi-polynomials $L(\sigma,t)$ and $L(\mathring{\sigma},t)$ satisfy the functional equation 
\begin{equation}
L(\sigma^{\circ},-t)=(-1)^{\dim\sigma}L(\sigma,t),\quad t\in\R.
\end{equation}
\end{thm}

With a method of Macdonald~\cite{macdonald1971polynomials}, which is generalized by Chen~\cite{chen1998weight}, the reciprocity law for a simplex can be extended to an arbitrary rational polytope, and more generally to a rational simplicial (polytopal) manifold $M$ with or without boundary (see Chen~\cite{chen2009gauss}), that is, $M$ is a cell complex, whose cells are rational polytopes and geometric realization is a manifold with or without boundary. For brevity, we state the result here only for rational polytopes.

\begin{thm}[Reciprocity law for rational polytope]\label{reciprocity law for polytope}
For $P$ each rational polytope of $\R^N$ with denominator $d$, the real Ehrhart quasi-polynomials $L(P,t)$ and $L(\mathring{P},t)$ satisfy the functional equation
\begin{equation}
L(P^{\circ},-t)=(-1)^{\dim P}L(P,t),\quad t\in\R.
\end{equation}
\end{thm}

The reciprocity law highlights a global symmetry for lattice counts on manifolds. A closely related phenomenon is the Ehrhart-Macdonald reciprocity for coefficient functions, first observed by Ehrhart and refined by Macdonald~\cite{macdonald1971polynomials} for integer dilations of integral polytopes. Extending this perspective, we show that the real Ehrhart coefficients $c_k(\sigma,t)$ and $c_k(\mathring{\sigma},t)$ satisfy an analogous relation. For clarity we state the result for simplices; the general case need more work and is not the main point of this paper, but it can be obtained by the same method of Macdonald~\cite{macdonald1971polynomials} and Chen~\cite{chen1998weight,chen2009gauss}.

\begin{prop}[Reciprocity law for coefficients] \label{reciprocity law for coefficients}
Let $\sigma$ be an $n$-simplex with denominator $d$. Then for all $t \in \R$, we have
\begin{equation}
c_k(\sigma,-t)=(-1)^{n-k} c_k(\mathring{\sigma},t).
\end{equation}
\end{prop}

Moreover, $c_k(P,t)$ and $c_k(\mathring{P},t)$ are piecewise polynomial. Eva Linke~\cite{linke2011rational} proved this in the rational case by showing that if a quasi-polynomial $f(t)=\sum_{k=0}^m c_k(t)t^k$ is constant on a interval, then the coefficient functions $c_k(t)$ are polynomials of degree $n-k$ and satisfy the derivative relation $c_k(t)'=-(k+1)c_{k+1}(t)$ on that interval. Baldoni et al.~\cite{baldoni2013intermediate} established the real case using intermediate sums that interpolate between discrete sums and integrals.

Our approach handles rational polytopes of arbitrary dimension and works directly with real dilations, thereby uncovering additional properties and recovering the rational case as a special case.

\begin{prop}\label{piecewise}
Let $P$ be a rational $n$-polytope of $\R^N$ with denominator $d$. If the vertices of $P$ are linearly dependent, then $c_k(P,t)$ and $c_k(\mathring{P},t)$ are piecewise polynomial with period $d$; on each piece they have degree $n-k$ and leading coefficient $(-1)^{n-k}\binom{n}{k}\vol(P)$. Otherwise, both $c_k(P,t)$ and $c_k(\mathring{P},t)$ vanish almost everywhere.
\end{prop}

The piecewise polynomial structure of the coefficient functions allows us to differentiate them on each piece. This leads to a simple formula relating consecutive coefficient functions:

\begin{prop}\label{derivative}
Let $P$ be a rational $n$-polytope of $\R^N$ with denominator $d$. Then, on each interval of piecewise polynomiality as described in \eqref{intervals}, we have
\begin{align*}
\frac{d}{dt}c_k(P,t)&=-(k+1)c_{k+1}(P,t),\quad\frac{d}{dt}c_k(\mathring{P},t)=-(k+1)c_{k+1}(\mathring{P},t).
\end{align*}
\end{prop}
At the end, we present an explicit two-dimensional example to illustrate the real Ehrhart quasi-polynomial. The paper is organized as follows. Section~2 derives the real Ehrhart quasi-polynomial together with its reciprocity law and contains the proofs of Theorems~\ref{main1} and~\ref{reciprocity law for Ehrhart function}. Section~3 studies the coefficient functions $c_k(P,t)$ and $c_k(\mathring{P},t)$ and establishes Propositions~\ref{reciprocity law for coefficients} and \ref{piecewise}. Section~4 works out a detailed example that is compatible with the previous results.

\section{Real Ehrhart quasi-polynomials}

Let $\R^N$ be an $N$-dimensional Euclidean space with integral lattice $\Z^N$. For a polytope $P\subseteq\R^N$, denote by $P^{\circ}=P\setminus\partial P$ the interior of $P$ in the affine span of $P$, known as a {\em relatively open polytope}. We call $P$ {\em integral} (resp. {\em rational}) if all vertices have integer (resp. rational) coordinates. The {\em dilation} of $P$ by a real number $t\ge 0$ is defined as $tP=\{tx:x\in P\}$, and the dilation of $\mathring{P}$ by real number $t>0$ is defined as $t\mathring{P}=\{tx:x\in\mathring{P}\}$.

Let $P$ be a rational polytope of $\R^N$. We often consider the counting functions
\begin{align}
L(P,t)&=\#(tP\cap\Z^N),  \qquad t\geq 0,\\
L(\mathring{P},t)&=\#(t\mathring{P}\cap\Z^N), \qquad t>0.
\end{align}
We define {\em denominator} of a rational polytope $P$ to be the smallest positive real number such that $dP$ is an integral polytope. It is easy to see that the denominator of $P$ is given by 
\[
d=\frac{\lcm(P)}{\gcd(P)},
\]
where $\gcd(P)$ (resp. $\lcm(P)$) denotes the greatest common divisor of the numerators (resp. least common multiple of the denominators) of the vertex coordinates written in reduced form. This convention differs from the traditional one, defined as the smallest positive integer $d$ with $dP$ integral, which is $\lcm(P)$ without dividing by $\gcd(P)$.

Every rational polytope can be decomposed into a disjoint union of relatively open rational simplices, it is crucial to find concrete patterns for the counting functions $L(\sigma,t)$ with $t\geq 0$ and $L(\mathring{\sigma},t)$ with $t>0$ of rational simplices $\sigma$. To obtain explicit formulas, we modify the methods of Chen~\cite{chen2005ehrhart,chen2002lattice}, which are rooted originally from Macdonald's idea~\cite{macdonald1963volume,macdonald1971polynomials}, to real dilations of $\sigma$. The following definition originates from Stanley's terminology~\cite{stanley2011enumerative} of determined sets.

\begin{defn}\label{DeterminedSets}
Let $\sigma$ be a rational $n$-simplex of $\R^N$ with vertices $\alpha_0,\alpha_1,\dots,\alpha_n$ and denominator $d$. The {\em determined sets of $\sigma$ at level $\ell\in\R$} is defined by
\begin{align}
D(\sigma,\ell)&=\bigg\{\sum_{j=0}^n u_j \alpha_j\in{\Bbb Z}^N\:\Big|\: 0\leq u_j<d,\ \sum_{j=0}^{n} u_j=\ell\bigg\}, \label{DL}\\
\bar D(\sigma,\ell)&=\bigg\{\sum_{j=0}^nv_j\alpha_j\in{\Bbb Z}^N\:\Big|\: 0<v_j\leq d,\ \sum_{j=0}^{n} v_j=\ell\bigg\}. \label{barDL} 
\end{align}
\end{defn}
The determined sets $D(\sigma,\ell)$ and $\bar D(\sigma,\ell)$ are empty for real numbers $\ell\notin[0,(n+1)d)$ and $\ell\notin(0,(n+1)d]$, respectively. The following proposition captures the basic behavior of these sets, which is intuitively true and is helpful for proving the main results of this paper.

\begin{prop}\label{step-function}
For each $\sigma$ rational $n$-simplex with denominator $d$, the counting functions $\#D(\sigma,\ell)$ and $\#\bar D(\sigma,\ell)$ of real variable $\ell$  are piecewise constant functions.
\end{prop}
\begin{proof}

Let $a$ and $b$ be nonnegative real numbers such that $a+b=(n+1)d$.
There is a one-to-one correspondence between $D(\sigma,a)$ and $\bar D(\sigma,b)$, sending $\sum_{j=0}^n u_j \alpha_j$ to $\sum_{j=0}^n (d-u_j) \alpha_j$. This is well-defined because $\sum_{j=0}^nu_j\alpha_j$ is a point of ${\Bbb Z}^N$ if and only if $\sum_{j=0}^n(d-u_j)\alpha_j$ is a point of ${\Bbb Z}^N$. Consequently
\begin{equation}\label{correspondence}
\#D(\sigma,\ell)=\#\bar D(\sigma,(n+1)d-\ell),
\end{equation}
and it suffices to prove the conclusion for $\# D(\sigma,\ell)$.

Let $A=(a_{ij})_{N\times(n+1)}$ be the matrix with columns $\alpha_0,\ldots,\alpha_n$, set $\bm u=(u_0,\ldots,u_n)^T\in[0,d)^{n+1}$, and write $\bm b=A\bm u=(b_1,b_2,\ldots,b_N)^T\in\Col A$. Then
\[
\|\bm b\|_\infty=\|A\bm u\|_\infty\leq\|A\|_\infty\|\bm u\|_\infty<\|A\|_\infty d=:M.
\]
Hence $D(\sigma,\ell)\subseteq\Z^N\cap[-M,M]^N\cap\Col A$ for every $\ell\in[0,(n+1)d)$. Fix $\bm b$ in this finite set. Determining the values of $\ell$ such that $\bm b\in D(\sigma,\ell)$, which amounts to finding the range of $\ell(\bm u)=\bm c^T\bm u$ (with $\bm c=(1,\ldots,1)^T$) under the constraints
\[
A\bm u=\bm b,\qquad 0\le u_j<d.
\]
The feasible set is the intersection of the affine solution space $W=\{\bm u\in\R^{n+1}:A\bm u=\bm b\}$ with the cube $[0,d)^{n+1}$. This intersection is convex (possibly empty), so the image of the linear functional $\ell(\bm u)=\bm c^T\bm u$ restricted to it is an interval $I_{\bm b}$, which may degenerate to a point or vanish entirely. Consequently,
\[
\#D(\sigma,\ell)=\sum_{\bm b\in\Z^N\cap[-M,M]^N}1_{I_{\bm b}}(\ell)
\]
is a step function of $\ell$, proving the proposition.
\end{proof}

To obtain an explicit formula for $\#D(\sigma,\ell)$ we must determine the interval $I_{\bm b}$ for each $\bm b\in \Z^N\cap[0,M)^N\cap\Col A$ (although this search region can be reduced in practice). The condition $\bm b\in D(\sigma,\ell)$ is equivalent to the existence of $\bm u_0\in[0,d)^{n+1}$ with $A\bm u_0=\bm b$ and $\bm c^T \bm u_0=\ell$, namely
\begin{align}\label{matrix equation for determined sets}
\binom{A}{\bm c^T}\bm u_0=\binom{\bm b}{\ell}.
\end{align}
Let $B=(A^T\ \bm c)^T$. Since $\alpha_0,\alpha_1,\dots,\alpha_n$ are vertices of an $n$-simplex, they are affinely independent. Therefore
\[
\text{rank}B=\text{rank}\begin{pmatrix}
    \alpha_0&\alpha_1&\dots&\alpha_n\\
    1&1&\dots&1
\end{pmatrix}=n+1.
\]
So $B$ has full column rank and therefore admits a left inverse $B^{-1}_L=(B^TB)^{-1}B^T$. 

If $\alpha_0,\ldots,\alpha_n$ are linearly independent, the system $A\bm u=\bm b$ has the unique solution $\bm u_0=(A^TA)^{-1}A^T\bm b$ (because $\bm b\in\Col A$). When $\bm u_0\in[0,d)^{n+1}$ we obtain $I_{\bm b}=\{\bm c^T \bm u_0\}$; otherwise $I_{\bm b}=\varnothing$.

If the vertices are linearly dependent, so that $\operatorname{rank}A=n$, pick any particular solution $\bm u_*$ to $A\bm u=\bm b$. Column operations on the augmented matrix of \eqref{matrix equation for determined sets} then yield
\[
\operatorname{rank}\begin{pmatrix}
A & \bm b\\
\bm c^T & \ell
\end{pmatrix}=\operatorname{rank}\begin{pmatrix}
A & 0\\
\bm c^T & \ell-\bm c^T\bm u_*
\end{pmatrix}=n+1,
\]
so the system \eqref{matrix equation for determined sets} is consistent for every $\ell$ and has the unique solution $\bm u_0=B^{-1}_L\binom{\bm b}{\ell}$. Consequently $\bm b\in D(\sigma,\ell)$ precisely when $\bm u_0\in[0,d)^{n+1}$. Equivalently, $I_{\bm b}$ consists of those $\ell$ that satisfy
\[
0\le \sum_{j=1}^{N} b_{ij}b_j+b_{i,N+1}\ell<d,\qquad i=0,1,\ldots,n,
\]
where $b_{ij}$ denotes the $(i,j)$-entry of $B^{-1}_L$. This observation inspires the following routine for computing $\#D(\sigma,\ell)$:

\begin{algorithm}[H]
\caption{Evaluate $\#D(\sigma,\ell)$}
\label{algorithm}
\LinesNumbered
\KwIn {$A=(\alpha_0,\ldots,\alpha_n)$, dimensions $n,N$, denominator $d$.}
\KwOut {$\# D(\sigma,\ell)=\sum_{b\in S}1_{I_{\bm b}}(l)$}
Compute $M_i=\sum_{j=0}^n|a_{ij}|d$ and set $S=\Z^N\cap([0,M_1]\times\cdots\times[0,M_N])\cap\Col A$\;
\If{$\operatorname{rank}A=n+1$}
{
    \For{$\bm b\in S$}    
    { 
    calculate $\bm u_0=(A^TA)^{-1}A^T\bm b$\;
    If $u_0\in[0,d)^{n+1}$, set $I_{\bm b}=\{\bm c^T \bm u_0\}$; else, set $I_{\bm b}=\emptyset$.
    }
}
\ElseIf{$\operatorname{rank}A=n$}
{
Compute the left inverse $B_L^{-1}=(b_{ij})$ of $(A^T\ \bm c)^T$ with $\bm c=(1,\ldots,1)^T$\;
\For{$\bm b\in S$}    
    { 
	calculate $I_{\bm b}$ by solving the inequalities $0\le \sum_{j=1}^{N} b_{ij}b_j+b_{i,N+1}\ell<d$ with respect to $\ell$.
    }
}
Return $\sum_{b\in S}1_{I_{\bm b}}(\ell)$.
\end{algorithm}

These calculations decompose the lattice-counting problem for $L(\sigma,t)$ and $L(\mathring{\sigma},t)$ into a summation over admissible levels $\ell$. Chen's lattice-point enumeration method~\cite{chen2000counting} unifies this decomposition: the argument for integer dilations simultaneously yields a closed formula for real dilations, as stated next.

\begin{prop}\label{Prop:Chen}
Let $\sigma$ be a rational $n$-simplex with denominator $d$. Then
\begin{align}
\#(t\sigma\cap\Z^N)&=\sum_{a\in\left(\frac{t}{d}-n-1,\frac{t}{d}\right]\cap\Z}\#D(\sigma;t-da) \binom{a+n}{n},\quad t\in\R_{\geq 0},
\label{eq:closed-simplex}\\
\#(t\mathring{\sigma}\cap\Z^N)&=\sum_{b\in\left[\frac{t}{d}-n-1,\frac{t}{d}\right)\cap\Z} \#\bar D(\sigma;t-db) \binom{b+n}{n},\quad t\in\R_{>0}.
\label{eq:open-simplex}
\end{align}
\end{prop}
\begin{proof}
Every point $\alpha\in t\sigma$ admits a unique representation
\[
\alpha=\sum_{j=0}^n a_j t\alpha_j,\qquad a_j\ge 0,\ \sum_{j=0}^n a_j=1,
\]
and similarly each $\beta\in t\mathring{\sigma}$ has
\[
\beta=\sum_{j=0}^n b_j t\alpha_j,\qquad b_j>0,\ \sum_{j=0}^n b_j=1.
\]
Dividing $a_jt$ and $b_jt$ by $d$ yields unique decompositions $a_jt=u_j+k_jd$ and $b_jt=v_j+l_jd$ with $0\le u_j<d$, $0<v_j\le d$, and $k_j,l_j\in\Z_{\ge0}$. Substituting these decompositions into the above expressions for $\alpha$ and $\beta$ gives
\begin{align}\label{eq:division of coefficients}
\alpha=\sum_{j=0}^n u_j\alpha_j+\sum_{j=0}^n k_j d\alpha_j,\quad \beta=\sum_{j=0}^n v_j\alpha_j+\sum_{j=0}^n l_j d\alpha_j.
\end{align}
Because the vectors $d\alpha_j$ are integral, $\alpha$ (resp. $\beta$) is a lattice point precisely when $\sum u_j\alpha_j$ (resp. $\sum v_j\alpha_j$) lies in $\Z^N$, equivalently, in $D(\sigma;\sum u_j)$ (resp. $\bar D(\sigma;\sum v_j)$).

Set $a=\sum_{j=0}^n k_j$. Then $\sum u_j=t-da$, and the bounds $0\le u_j<d$ imply
\[
a=\frac{t}{d}-\frac{1}{d}\sum_{j=0}^n u_j\in\Bigl(\frac{t}{d}-n-1,\frac{t}{d}\Bigr]\cap\Z.
\]
By the uniqueness of the representation \eqref{eq:division of coefficients}, the lattice points of $t\sigma$ are naturally grouped by $a$ into disjoint sets 
\[
\left\{\sum u_j\alpha_j+\sum k_j d\alpha_j:\sum u_j\alpha_j\in D(\sigma;t-da),\ \sum k_j=a\right\}
\]
as $a$ ranges over all admissible integers.

Fix one such $a$. The number of nonnegative integer tuples $(k_0,\ldots,k_n)$ summing to $a$ equals $\binom{a+n}{n}$, and each such tuple contributes exactly $\#D(\sigma;t-da)$ lattice points $\alpha$ of the form $\sum u_j\alpha_j+\sum k_j d\alpha_j$. By the uniqueness of the representation \eqref{eq:division of coefficients}, lattice points arising from different tuples are distinct, so the contribution of this $a$ to $L(\sigma,t)$ is $\#D(\sigma;t-da)\binom{a+n}{n}$. Summing these contributions over all admissible $a$ yields \eqref{eq:closed-simplex}.

For interior dilations we instead set $b=\sum_{j=0}^n l_j$. The relation $\sum v_j=t-db$ forces $b\in[\frac{t}{d}-n-1,\frac{t}{d})\cap\Z$, and applying the same counting argument produces \eqref{eq:open-simplex}.
\end{proof}

From this perspective, \eqref{eq:closed-simplex} and \eqref{eq:open-simplex} express the counting functions $\#(t\sigma\cap\Z^N)$ and $\#(t\mathring{\sigma}\cap\Z^N)$ as polynomial expansions in the binomial basis $\binom{a+n}{n}$ and $\binom{b+n}{n}$, each term encoding a determined set. To analyze the coefficient functions that appear in these expansions, it is convenient to adopt the notation used by Stanley~\cite{stanley2011enumerative}.

For every real number $x$, define its fractional parts relative to the floor and ceiling by
\[
\{x\}:=x-\lfloor x\rfloor,  \text{ and }  \langle x\rangle:=x-\lceil x\rceil.
\]
Then $0\le \{x\}<1$ and $0\le \langle x\rangle<1$, with $\{x\}=\langle x\rangle=0$ precisely when $x\in\Z$ and $\{x\}-\langle x\rangle=1$ otherwise. We also have $\{-x\}=-\langle x\rangle$ for all $x$, and both are periodic functions of $x$ with period $1$. 

Recall the definition of the binomial coefficient: for $x,y\in\R$ and $n\in\Z_{\geq 0}$,
\begin{align}
\binom{x+y+n}{n}
&=\frac{1}{n!}(x+y+1)(x+y+2)\cdots(x+y+n)\label{Binomial-Coefficient-Definition}\\
&=\frac{1}{n!}\sum_{k=0}^n x^k\,s_{n-k}(y+1,\ldots,y+n),
\label{Binomial-Coefficient-Elemetary-Symmetry-Poly}
\end{align}
where $s_j$ is the $j$th elementary symmetric polynomial in $n$ variables, that is,
\[
s_j(y_1,\ldots,y_n)=\sum_{1\le i_1<\cdots<i_j\le n}y_{i_1}\cdots y_{i_j}, \quad j=0,1,\ldots,n,
\]
We now use this identity to extract the coefficient functions promised in Theorem~\ref{main1}.

\begin{proof}[Proof of Theorem~\ref{main1}]
Let $\sigma$ be an $n$-simplex with denominator $d$. For any real $t$,
\[
\Bigl(\tfrac{t}{d}-n-1,\tfrac{t}{d}\Bigr]\cap\Z=
\Bigl\{\bigl\lfloor\tfrac{t}{d}\bigr\rfloor,\bigl\lfloor\tfrac{t}{d}\bigr\rfloor-1,\ldots,
\bigl\lfloor\tfrac{t}{d}\bigr\rfloor-n\Bigr\}.
\]
Setting $a=\lfloor\tfrac{t}{d}\rfloor-j$ with $0\le j\le n$ and using $\lfloor x\rfloor=x-\{x\}$ yields
\[
a=\frac{t}{d}-j-\Bigl\{\frac{t}{d}\Bigr\},\qquad t-da=dj+d\Bigl\{\frac{t}{d}\Bigr\}.
\]
Substituting into \eqref{eq:closed-simplex} gives
\begin{equation}\label{Lsigma}
\#(t\sigma\cap\Z^N)=\sum_{j=0}^n \#D\Bigl(\sigma, dj+d\Bigl\{\frac{t}{d}\Bigr\}\Bigr)\binom{\frac{t}{d}-j-\{\frac{t}{d}\}+n}{n}.
\end{equation}
Applying \eqref{Binomial-Coefficient-Elemetary-Symmetry-Poly} with $x=\tfrac{t}{d}$ and $y=-j-\{\tfrac{t}{d}\}$ gives
\[
\binom{\frac{t}{d}-j-\{\frac{t}{d}\}+n}{n}
=\frac{1}{n!}\sum_{k=0}^n\Bigl(\frac{t}{d}\Bigr)^k
s_{n-k}\Bigl(1-j-\Bigl\{\frac{t}{d}\Bigr\},\ldots,n-j-\Bigl\{\frac{t}{d}\Bigr\}\Bigr).
\]
Inserting this expansion into \eqref{Lsigma} and collecting the coefficients of each power of $t$ gives $\#(t\sigma\cap\Z^N)=\sum_{k=0}^n c_k(\sigma,t)t^k$ with
\begin{equation}\label{coefficient-of-closed-simplex}
c_k(\sigma,t)=\frac{1}{n!d^k}\sum_{j=0}^n \#D\Bigl(\sigma, dj+d\Bigl\{\frac{t}{d}\Bigr\}\Bigr)
s_{n-k}\Bigl(1-j-\Bigl\{\frac{t}{d}\Bigr\},\ldots,n-j-\Bigl\{\frac{t}{d}\Bigr\}\Bigr).
\end{equation}
Since the fractional part $\{\tfrac{t}{d}\}$ has period $d$ and is defined for all real $t$, each coefficient $c_k(\sigma,t)$ inherits the same period and is well-defined on $\R$.

The open case is analogous. The integers in $\bigl[\tfrac{t}{d}-n-1,\tfrac{t}{d}\bigr)$ are
\[
\Bigl\{\bigl\lceil\tfrac{t}{d}\bigr\rceil-1,\ldots,\bigl\lceil\tfrac{t}{d}\bigr\rceil-(n+1)\Bigr\}.
\]
With $b=\lceil\tfrac{t}{d}\rceil-j$ ($1\le j\le n+1$) and $\lceil x\rceil=x-\langle x\rangle$ we obtain
\[
b=\frac{t}{d}-j-\Bigl\langle\frac{t}{d}\Bigr\rangle,\qquad t-db=d\Bigl(j+\Bigl\langle\frac{t}{d}\Bigr\rangle\Bigr).
\]
Using these expressions in \eqref{eq:open-simplex} we obtain
\begin{equation}\label{Lsigmaopen}
\#(t\mathring{\sigma}\cap\Z^N)=\sum_{j=1}^{n+1}\#\bar D\Bigl(\sigma,dj+d\Bigl\langle\frac{t}{d}\Bigr\rangle\Bigr)
\binom{\frac{t}{d}-j-\langle\frac{t}{d}\rangle+n}{n},
\end{equation}
Applying \eqref{Binomial-Coefficient-Elemetary-Symmetry-Poly} again gives $\#(t\mathring{\sigma}\cap\Z^N)=\sum_{k=0}^n c_k(\mathring{\sigma},t)t^k$ with
\begin{equation}\label{coefficient-of-open-simplex}
c_k(\mathring{\sigma},t)=\frac{1}{n!d^k}\sum_{j=1}^{n+1}\#\bar D\Bigl(\sigma,dj+d\Bigl\langle\frac{t}{d}\Bigr\rangle\Bigr)
s_{n-k}\Bigl(1-j-\Bigl\langle\frac{t}{d}\Bigr\rangle,\ldots,n-j-\Bigl\langle\frac{t}{d}\Bigr\rangle\Bigr).
\end{equation}
Since $\langle\tfrac{t}{d}\rangle$ is defined for all real $t$ and has period $d$, each coefficient $c_k(\mathring{\sigma},t)$ is likewise well defined on $\R$ and $d$-periodic.

Now consider a general rational polytope $P$ of denominator $d$. Decompose $P$ into finitely many pairwise disjoint open simplices $\mathring{\sigma}_1,\ldots,\mathring{\sigma}_s$, each with denominator $d$. Then
\begin{align*}
\#(tP\cap\Z^N)
&=\sum_{i=1}^s \#(t\mathring{\sigma}_i\cap\Z^N)\\
&=\sum_{i=1}^s\sum_{k=0}^{\dim\mathring{\sigma}_i} c_k(\mathring{\sigma}_i,t)t^k\\
&=\sum_{k=0}^n\Bigl(\sum_{\dim\mathring{\sigma}_i\ge k} c_k(\mathring{\sigma}_i,t)\Bigr)t^k.
\end{align*}
Define $c_k(P,t)=\sum_{\dim\mathring{\sigma}_i\ge k} c_k(\mathring{\sigma}_i,t)$. Then each $c_k(P,t)$ inherits period $d$, and hence $L(P,t)=\sum_k c_k(P,t)t^k$ is a real quasi-polynomial on $\R$ whose restriction to $[0,\infty)$ agrees with the counting function $\#(tP\cap\Z^N)$. The same argument applied to $\mathring{P}$ proves the corresponding statement for $L(\mathring{P},t)$.
\end{proof}

The proof above yields concrete formulas for simplex coefficients in terms of determined sets: by periodicity, \eqref{coefficient-of-closed-simplex} and \eqref{coefficient-of-open-simplex} coincide with the coefficient functions $c_k(\sigma,t)$ and $c_k(\mathring{\sigma},t)$ announced in Theorem~\ref{main1} on all of $\R$. These are simply rewritten as \eqref{concrete coefficient of closed simplex} and \eqref{concrete coefficient of open simplex}.

To compute quasi-polynomial expansions for $L(\sigma,t)$ and $L(\mathring{\sigma},t)$, one evaluates $\#D(\sigma,\ell)$ and $\#\bar D(\sigma,\ell)$ via Algorithm~\ref{algorithm}, then substitutes into the coefficient formulas \eqref{coefficient-of-closed-simplex} and \eqref{coefficient-of-open-simplex}, or directly into the binomial formulas \eqref{real close binomial} and \eqref{real open binomial} (see Section~\ref{examples}).

Examining the argument in the proof of Theorem~\ref{main1} in reverse, we see that the same binomial expressions for $L(\sigma,t)$ and $L(\mathring{\sigma},t)$ remain valid for every real $t$, not only for $t\ge 0$ and $t>0$. The key point is that the coefficient functions are already defined on all of $\R$ and satisfy the same periodic formulas throughout.

\begin{prop}\label{binomial formula for real dilations}
Let $\sigma$ be a rational simplex of denominator $d$. For every real $t$,
\begin{align}
L(\sigma,t)&=\sum_{a\in\left(\frac{t}{d}-n-1,\frac{t}{d}\right]\cap\Z}\#D(\sigma;t-da) \binom{a+n}{n},\quad t\in\R,\label{real close binomial}\\
L(\mathring{\sigma},t)&=\sum_{b\in\left[\frac{t}{d}-n-1,\frac{t}{d}\right)\cap\Z} \#\bar D(\sigma;t-db) \binom{b+n}{n},\quad t\in\R.\label{real open binomial}
\end{align}
\end{prop}
\begin{proof}
For any real $t$, substituting \eqref{coefficient-of-closed-simplex} into $L(\sigma,t)=\sum_{k=0}^n c_k(\sigma,t)t^k$ gives
\[
L(\sigma,t)=\sum_{k=0}^n\frac{1}{n!d^k}\sum_{j=0}^n \#D\Bigl(\sigma, dj+d\Bigl\{\frac{t}{d}\Bigr\}\Bigr)
s_{n-k}\Bigl(1-j-\Bigl\{\frac{t}{d}\Bigr\},\ldots,n-j-\Bigl\{\frac{t}{d}\Bigr\}\Bigr)t^k.
\]
Reordering the summation and applying \eqref{Binomial-Coefficient-Elemetary-Symmetry-Poly} with $x=\tfrac{t}{d}$ and $y=-j-\{\tfrac{t}{d}\}$, we obtain
\[
L(\sigma,t)=\sum_{j=0}^n \#D\Bigl(\sigma, dj+d\Bigl\{\frac{t}{d}\Bigr\}\Bigr)\binom{\frac{t}{d}-j-\{\frac{t}{d}\}+n}{n}
=\sum_{a\in\left(\frac{t}{d}-n-1,\frac{t}{d}\right]\cap\Z}\#D(\sigma;t-da) \binom{a+n}{n}.
\]
Starting instead from \eqref{coefficient-of-open-simplex}, the same argument yields the corresponding formula for $L(\mathring{\sigma},t)$. Therefore both binomial formulas hold for all real dilations.
\end{proof}

Here $\binom{x}{n}$ denotes the generalized binomial coefficient defined by \eqref{Binomial-Coefficient-Definition} for real $x$ and $n\in\Z_{\geq0}$. Moreover, for every real $t$ and every admissible
\[
a\in\left(\frac{t}{d}-n-1,\frac{t}{d}\right]\cap\Z,\qquad b\in\left[\frac{t}{d}-n-1,\frac{t}{d}\right)\cap\Z,
\]
we have $t-da\in[0,(n+1)d)$ and $t-db\in(0,(n+1)d]$. Hence the step functions $\#D(\sigma;t-da)$ and $\#\bar D(\sigma;t-db)$ are well-defined, and the above binomial formulas make sense for all $t\in\R$.

The reciprocity law in Theorem~\ref{reciprocity law for Ehrhart function} now follows from the binomial identity
\begin{align}\label{eq:binomial reciprocity}
\binom{n+x}{n}=(-1)^n\binom{-x-1}{n},
\end{align}
which relates the formulas for $L(\sigma,t)$ and $L(\mathring{\sigma},t)$, as stated next.

\begin{proof}[Proof of Theorem~\ref{reciprocity law for Ehrhart function}]
We first prove the statement for a rational simplex $\sigma$. By \eqref{correspondence} in the proof of Proposition~\ref{step-function}, for every real $t$ we have
\begin{equation*}
L(\sigma^{\circ},-t)=\sum_{b\in\left[-\frac{t}{d}-n-1,\:-\frac{t}{d}\right)\cap{\Z}} \# D(\sigma;(n+1)d-(t-db)) \binom{b+n}{n}.
\end{equation*}
Set $l=-n-1-b$. Then
\[
b\in\left[-\frac{t}{d}-n-1,\:-\frac{t}{d}\right)\cap\Z
\quad\Longleftrightarrow\quad
l\in\left(\frac{t}{d}-n-1,\frac{t}{d}\right]\cap\Z.
\]
Substituting this change of variables into the above expression for $L(\sigma^{\circ},-t)$, we obtain
\[
L(\sigma^{\circ},-t)=\sum_{l\in\left(\frac{t}{d}-n-1,\:\frac{t}{d}\right]\cap{\Z}} \# D(\sigma;-t-dl) \binom{-l-1}{n}.
\]
Applying the reciprocity identity \eqref{eq:binomial reciprocity} for binomial coefficients, we obtain
\[
L(\sigma^{\circ},-t)=(-1)^n\sum_{l\in\left(\frac{t}{d}-n-1,\:\frac{t}{d}\right]\cap{\Z}} \# D(\sigma;-t-dl) \binom{n+l}{n}.
\]
The sum on the right is exactly $(-1)^nL(\sigma,t)$, and hence $L(\sigma^{\circ},-t)=(-1)^nL(\sigma,t)$.
\end{proof}

With the reciprocity law for rational simplices established, the general case follows by the same method used in the proof of the classical Ehrhart reciprocity law for rational polytopes, which similarly relies on the simplex case. The remainder of this section carries out this strategy to prove the reciprocity law for rational polytopes.

\begin{proof}[Proof of Theorem~\ref{reciprocity law for polytope}]
Let $K$ be a simplicial decomposition of $P$ with $\emptyset\notin K$. We begin by applying the simplex reciprocity law to each relatively open simplex in $K$:
\[
L(P,-t)=\sum_{\sigma\in K}L(\mathring{\sigma},-t)=\sum_{\sigma\in K}(-1)^{\dim\sigma}L(\sigma,t).
\]
Expanding each $L(\sigma,t)$ into the sum over the relatively open faces of $\sigma$, we obtain
\begin{align*}
L(P,-t)=\sum_{\sigma\in K}(-1)^{\dim\sigma}\sum_{\tau\leq\sigma}L(\mathring{\tau},t)=\sum_{\tau\in K}L(\mathring{\tau},t)\sum_{\tau\leq\sigma\in K}(-1)^{\dim\sigma}.
\end{align*}
Thus the problem reduces to evaluating, for each face $\tau\in K$, the inner alternating sum over all simplices of $K$ that contain $\tau$. This sum is precisely the Euler characteristic of the star $St(\tau)$ of $\tau$ in $K$, so
\begin{align*}
\sum_{\tau\leq\sigma\in K}(-1)^{\dim\sigma}=
\begin{cases}
(-1)^{\dim P} & \text{if }\tau\nsubseteq\partial P,\\
0 & \text{if }\tau\subseteq\partial P.
\end{cases}
\end{align*}
Indeed, if $\tau$ is not contained in $\partial P$, then $St(\tau)$ is a ball of dimension $\dim P-\dim\tau$, whereas if $\tau\subseteq\partial P$, then $St(\tau)$ is a boundary ball and its Euler characteristic is zero. Substituting these values back into the previous expression yields
\[
L(P,-t)=\sum_{\tau\in K,\tau\nsubseteq\partial P}L(\mathring{\tau},t)(-1)^{\dim P}=(-1)^{\dim P}L(\mathring{P},t),
\]
which completes the proof.

\end{proof}

Theorem~\ref{main1} shows that the counting functions for rational simplices and rational polytopes are real quasi-polynomials, thereby extending the classical theory of integer dilations of integral polytopes. Moreover, Theorems \ref{reciprocity law for Ehrhart function} and \ref{reciprocity law for polytope} confirm that the reciprocity law extends to this broader context. We now turn to a more refined study of the coefficient functions.

\section{The coefficients of real Ehrhart quasi-polynomial}

In this section we study the coefficient functions of the real Ehrhart quasi-polynomials. We begin by rewriting the simplex coefficient formulas in a form convenient for proving reciprocity and piecewise polynomiality. Define the multivarible functions $\mathbf{s}^\sigma_k,\bar{\mathbf{s}}^\sigma_k:\R^{n+1}\to\R$ and varible functions $x_l(i,t),y_l(j,t):\R\to\R$
by \eqref{eq:multivariable functions for closed simplex}, \eqref{eq:multivariable functions for open simplex}, \eqref{eq:variable functions for closed simplex} and \eqref{eq:variable functions for open simplex}, respectively.
In terms of this notation, the formulas \eqref{coefficient-of-closed-simplex} and \eqref{coefficient-of-open-simplex} for the coefficient functions $c_k(\sigma,t)$ and $c_k(\mathring{\sigma},t)$ take the form
\begin{align}
c_k(\sigma,t)&=\frac{1}{n!d^k}\sum_{i=0}^{n}\mathbf{s}_{n-k}(\sigma;x_0(i,t),\dots,x_n(i,t)),\label{coefficient of closed simplex}\\
c_k(\mathring{\sigma},t)&=\frac{1}{n!d^k}\sum_{j=1}^{n+1}\bar{\mathbf{s}}_{n-k}(\sigma;y_0(j,t),\dots,y_n(j,t)).\label{coefficient of open simplex}
\end{align}

These formulas now allow us to establish the reciprocity law for simplex coefficients using symmetric function arguments and the correspondence between the determined sets, as stated next. An alternative proof could be given by applying the reciprocity law to quasi-polynomials and invoking the uniqueness of their coefficients, but this would first require establishing that uniqueness itself.

\begin{proof}[Proof of Proposition \ref{reciprocity law for coefficients}]
Since $\left\{-\frac{t}{d}\right\}=-\left\langle\frac{t}{d}\right\rangle$, we have for $0\leq l_1,l_2,i\leq n$ and $1\leq j\leq n+1$:
\[
x_{l_1}(i,-t)+y_{l_2}(j,t)=l_1+l_2-i-j.
\]
In particular, choosing $j=n+1-i$ and $l_2=n+1-l_1$ gives $x_{l_1}(i,-t)=-y_{n+1-l_1}(n+1-i,t)$. Using the symmetry and homogeneity of elementary symmetric polynomials, we then obtain
\begin{align*}
s_{n-k}(x_1(i,-t),\dots,x_n(i,-t))=(-1)^{n-k}s_{n-k}(y_1(j,t),\dots,y_n(j,t)).
\end{align*}
Observe that $-x_0(i,-t)-y_0(j,t)=n+1$. Combining with \eqref{correspondence} (from Proposition~\ref{step-function}) gives
\[
\#D\left(\sigma,-dx_0(i,-t)\right)=\#\bar{D}\left(\sigma,-dy_0(j,t)\right).
\]
Substituting both results into \eqref{eq:multivariable functions for closed simplex} and \eqref{eq:multivariable functions for open simplex} yields
\begin{align*}
\mathbf{s}_{n-k}(\sigma;x_0(i,-t),\dots,x_n(i,-t))=(-1)^{n-k}\bar{\mathbf{s}}_{n-k}(\sigma;y_0(j,t),\dots,y_n(j,t)).
\end{align*}
The multivariable forms in \eqref{coefficient of closed simplex} and \eqref{coefficient of open simplex} gives $c_k(\sigma,-t)=(-1)^{n-k}c_k(\mathring{\sigma},t)$, as claimed.
\end{proof}

The polytope case follows by applying the same argument as in the proof of Theorem~\ref{reciprocity law for polytope} but with $L(P,t)$ and $L(\mathring{P},t)$ replaced by $c_k(P,t)$ and $c_k(\mathring{P},t)$, respectively. Thus, for every rational polytope $P$ and every real $t$,
\[
c_k(P,-t)=(-1)^{\dim P-k}c_k(\mathring{P},t).
\]

To establish piecewise polynomiality of the coefficient functions, we first derive summation formulas for the determined sets. These formulas are essential for determining the degree and leading coefficients of each polynomial piece.

\begin{prop}\label{prop:volume counting}
Let $\sigma$ be a rational $n$-simplex of denominator $d$ with linearly dependent vertices. Then for all $t\in\R$,
\begin{align}
\sum_{i=0}^{n}\#D(\sigma,-dx_0(i,t))&=n!d^n\cdot\vol(\sigma),\label{eq:closed volume}\\
\sum_{j=1}^{n+1}\#\bar{D}(\sigma,-dy_0(j,t))&=n!d^n\cdot\vol(\sigma).\label{eq:open volume}
\end{align}
\end{prop}

\begin{proof}
The key observation is that, after integral translation, the determined sets sum to a fundamental domain of the simplex lattice.

\textit{Closed simplex case:} Since $-x_0(i,t)=i+\{\frac{t}{d}\}$ is periodic with period $d$, it suffices to verify \eqref{eq:closed volume} for $t\in[0,d)$, where $-dx_0(i,t)=t+di$. By separating the $\alpha_0$ coordinate in \eqref{DL}, we obtain
\[
D(\sigma,t+di)=di\alpha_0+t\alpha_0+R_i,
\]
where $R_i=\{\sum_{j=1}^n u_j(\alpha_j-\alpha_0): u_j\in[0,d),\ t+di-d<\sum_{j=1}^n u_j\leq t+di\}$. 

Since $di\alpha_0\in\Z^N$ and $R_i$ are disjoint for $0\le i\le n$, we obtain
\[
\varphi:\sum_{i=0}^n\#D(\sigma,t+di)=\sum_{i=0}^n\#\left((t\alpha_0+R_i)\cap\Z^N\right)=\#\left((t\alpha_0+R)\cap\Z^N\right),
\]
where $R:=\bigsqcup_{i=0}^n R_i=\{\sum_{j=1}^n u_j(\alpha_j-\alpha_0): u_j\in[0,d)\}$ is a fundamental region for the lattice $\Lambda=\{d\sum_{j=1}^n m_j(\alpha_j-\alpha_0):m_j\in\Z\}$. By definition of denominator, $\Lambda$ is a sublattice of $\Z^N$ with index $\vol(R)$. Thus $\#(R\cap\Z^N)=\vol(R)$, and by the volume formula for simplices, $\vol(R)=n!d^n\vol(\sigma)$, giving $\#(R\cap\Z^N)=n!d^n\vol(\sigma)$.

Since the vertices $\alpha_0,\alpha_1,\dots,\alpha_n$ are affinely independent and (by assumption) linearly dependent, we have $\alpha_0=\sum_{j=1}^n c_j(\alpha_j-\alpha_0)$ for some $c_j\in\R$. Using this relation, we construct a map
\[
\sum_{j=1}^n u_j d(\alpha_j-\alpha_0)\mapsto t\alpha_0+\sum_{j=1}^n \{u_j-tc_j\} d(\alpha_j-\alpha_0).
\]
It is straightforward to verify that this map induces a lattice-point bijection between $R$ and $t\alpha_0+R$. Thus,
\[
\sum_{i=0}^n\#D(\sigma,t+di)=\#(R\cap\Z^N)=n!d^n\vol(\sigma).
\]

\textit{Open simplex case:} For $t\in(0,d]$, we have $-dy_0(j,t)=t+dj-d$. As before, we can rewrite $\bar D(\sigma,t+dj-d)$ as $t\alpha_0+(dj-d)\alpha_0+\bar R_j$, where
\[
\bar R_j:=\{\sum_{l=1}^n u_l(\alpha_l-\alpha_0): u_l\in[0,d),\ t+dj-2d<\sum_{l=1}^n u_l\leq t+dj-d\}.
\]
Hence $\sum_{j=1}^{n+1}\#\bar{D}(\sigma,t+dj-d)=\#\left((t\alpha_0+\bigsqcup_{j=1}^{n+1}\bar R_j)\cap\Z^N\right)$. Analogously, the map
\[
\sum_{l=1}^n u_l d(\alpha_l-\alpha_0)\mapsto t\alpha_0+\sum_{l=1}^n (1-\{u_l-tc_l\}) d(\alpha_l-\alpha_0).
\]
induces a lattice-point bijection between $R$ and $t\alpha_0+\bigsqcup_{j=1}^{n+1}\bar R_j$. Thus,
\[
\sum_{j=1}^{n+1}\#\bar{D}(\sigma,t+dj-d)=\#(R\cap\Z^N)=n!d^n\vol(\sigma).
\]
\end{proof}

With these preparations, we now turn to the proof of piecewise polynomiality. By Proposition \ref{step-function}, $\#D(\sigma,\ell)$ and $\#\bar D(\sigma,\ell)$ are piecewise constant functions from $[0,(n+1)d)$ to $\Z_{\geq 0}$. Consequently, the coefficient functions $c_k(\sigma,t)$ and $c_k(\mathring{\sigma},t)$ inherit this piecewise structure and are therefore piecewise polynomials. The details follow.

\begin{proof}[Proof of Proposition \ref{piecewise}]
We establish piecewise polynomiality in stages: first for a relatively open simplex, then for general polytopes.

\textit{Step 1: Open simplex.} By periodicity, it suffices to prove that $c_k(\mathring{\sigma},t)$ is piecewise polynomial for $t\in(0,d]$. The formula \eqref{coefficient-of-open-simplex} expresses the coefficient as a weighted sum of elementary symmetric polynomials. Since each $\#\bar{D}(\sigma,t+dj-d)$ is a step function with finitely many jump points, we can partition $(0,d]$ into intervals $(r_i,r_{i+1})$ where each $\#\bar{D}(\sigma,t+dj-d)=C_{ij}$ is constant. On each interval,
\[
c_k(\mathring{\sigma},t)=\frac{1}{n!d^k}\sum_{j=1}^{n+1}C_{ij}\cdot s_{n-k}(y_1(j,t),\dots,y_n(j,t))
\]
is a polynomial of degree at most $n-k$ in $t$ (since $y_l(j,t)=l-j+1-\frac{t}{d}$ is linear in $t$).

For the leading term: $s_{n-k}(y_1(j,t),\dots,y_n(j,t))$ has leading coefficient $\binom{n}{n-k}(-\frac{1}{d})^{n-k}$ in $t^{n-k}$. Thus the coefficient of $t^{n-k}$ in $c_k(\mathring{\sigma},t)$ is
\[
\frac{(-1)^{n-k}}{n!d^n}\binom{n}{k}\sum_{j=1}^{n+1}C_{ij}.
\]

If $\sigma$ has linearly dependent vertices, Proposition~\ref{prop:volume counting} gives $\sum_{j=1}^{n+1}C_{ij}=n!d^n\vol(\sigma)$, so the leading coefficient is
\[
(-1)^{n-k}\binom{n}{k}\vol(\sigma)\neq 0.
\]
If $\sigma$ has linearly independent vertices, then $\#\bar{D}(\sigma,\ell)$ is zero on each interval by Algorithm~\ref{algorithm}, so $c_k(\mathring{\sigma},t)$ is zero on each interval and hence almost zero.

\textit{Step 2: General polytope.} For a polytope $P$ with simplicial decomposition into relatively open simplices $\mathring{\sigma}_1,\ldots,\mathring{\sigma}_s$, we have
\[
c_k(P,t)=\sum_{\dim\mathring{\sigma}_p\geq k}c_k(\mathring{\sigma}_p,t)
\]
is a finite sum of piecewise polynomials, hence piecewise polynomial. 

If $P$ has linearly independent vertices, every $\sigma_p$ has this property, so $c_k(P,t)$ is almost zero. If $P$ has linearly dependent vertices, then all $n$-dimensional simplices in the decomposition have this property. The leading term $t^{n-k}$ comes from these $n$-dimensional pieces:
\[
c_k(P,t)=\sum_{\dim\sigma_p=n}(-1)^{n-k}\binom{n}{k}\vol(\sigma_p)\cdot t^{n-k}+\text{lower order terms}.
\]
Hence, $c_k(P,t)$ has degree exactly $n-k$ with leading coefficient $(-1)^{n-k}\binom{n}{k}\vol(P)$.
\end{proof}

For a polytope $P$ with a simplicial decomposition into relatively open simplices $\mathring{\sigma}_1,\ldots,\mathring{\sigma}_s$, the pieces are given by the intervals
\begin{align}\label{intervals}
(r_0+qd,r_1+qd),\ (r_1+qd,r_2+qd),\ \ldots,\ (r_{m-1}+qd,r_m+qd),\quad q\in\Z,
\end{align}
where $0=r_0<r_1<\cdots<r_m=d$ denote the jump points of $\#\bar{D}(\sigma_p,t+dj-d)$ for $t\in(0,d]$.

We now establish the derivative relation between the coefficients; this relation is already evident from the leading terms. It follows from the symmetric function structure of the coefficient formulas and the linearity of differentiation. The details follow.

\begin{proof}[Proof of Proposition \ref{derivative}]
\textit{Simplex case.}
Since $c_k(\sigma,t)$ and $c_k(\mathring{\sigma},t)$ are periodic with period $d$, it suffices to prove the conclusion for $t\in[0,d)$. By the reciprocity law of coefficients, if the formula holds for $c_k(\sigma,t)$, then
\[
\frac{d}{dt}c_k(\mathring{\sigma},t)=(-1)^{n-k-1}\frac{d}{dt}c_k(\sigma,-t)=(-1)^{n-k-1}c_{k+1}(\sigma,-t)=c_{k+1}(\mathring{\sigma},t).    
\]
Thus, we reduce to proving the formula for $c_k(\sigma,t)$ on $[0,d)$.

For the elementary symmetric polynomial, we have
\[
\frac{\partial s_k(x_1,\dots,x_n)}{\partial x_j}=s_{k-1}(x_1,\dots,\hat{x}_j,\dots,x_n),
\]
where $\hat{x}_j$ denotes the omission of $x_j$. By the chain rule and the fact that $\frac{d}{dt}x_j(i,t)=-\frac{1}{d}$, we have
\begin{align*}
\frac{d}{dt}s_{k}(x_1(i,t),\dots,x_n(i,t))&=-\frac{1}{d}\sum_{j=1}^{n}s_{k-1}(x_1(i,t),\dots,\hat{x}_j(i,t),\dots,x_n(i,t)),
\end{align*}
Since each monomial of $k-1$ distinct variables appears exactly $n-k+1$ times in the above sum, we have
\[
\frac{d}{dt}s_{k}(x_1(i,t),\dots,x_n(i,t))=-\frac{1}{d}(n-k+1)s_{k-1}(x_1(i,t),\dots,x_n(i,t)).
\]
Hence, on each piece of $c_k(\sigma,t)$ with $\#D(\sigma,-dx_0(i,t))$ constant, we have
\begin{align*}
\frac{d}{dt}\mathbf{s}_k(\sigma;x_0(i,t),\dots,x_n(i,t))&=-\frac{1}{d}(n-k+1)\mathbf{s}_{k-1}(\sigma;x_0(i,t),\dots,x_n(i,t)).
\end{align*}
Substituting into formula \eqref{coefficient of closed simplex}, we obtain $\frac{d}{dt}c_k(\sigma,t)=-(k+1)c_{k+1}(\sigma,t)$.

\textit{General polytope case.}
For an arbitrary polytope $P$ with simplicial decomposition into relatively open simplices $\mathring{\sigma}_1,\ldots,\mathring{\sigma}_s$, we have $c_k(P,t)=\sum_{\dim\sigma_p\geq k}c_k(\sigma_p,t)$. By linearity of differentiation,
\[
\frac{d}{dt}c_k(P,t)=\sum_{\dim\sigma_p\geq k}\frac{d}{dt}c_k(\sigma_p,t)=\sum_{\dim\sigma_p\geq k}(-(k+1)c_{k+1}(\sigma_p,t))=-(k+1)c_{k+1}(P,t),
\]
where the last equality uses the convention that $c_{k+1}(\sigma_p,t)=0$ for $\dim\sigma_p<k+1$.
\end{proof}


\section{A detailed example}\label{examples}
We illustrate our methods with a concrete example. Consider the $3$-simplex $\sigma\subseteq\Z^3$ with vertices 
$$\alpha_0=(0,0,0), \quad \alpha_1=(1,1,0), \quad \alpha_2=(1,0,1), \quad \alpha_3=(0,1,1).$$
Define the binomial coefficient functions
\[
B_i(t):=\binom{t-\{t\}+i}{3}, \quad B'_j(t):=\binom{t-\langle t\rangle+j}{3}, \quad i,j\in\Z.
\]

By Proposition \ref{binomial formula for real dilations}, $L(\sigma,t)$ is given by the sum
\[
L(\sigma,t)=\sum_{i=0}^{3}\#D(\sigma,\{t\}+i)B_{3-i}(t).
\]
The key step is determining $\#D(\sigma,\ell)$ for each level $\ell$. We count all $\bm u\in[0,1)^4$ satisfying
\[
A\bm u\in\Z^3 \quad\text{and}\quad \bm c^T\bm u=\ell,
\]
where $A=(\alpha_0,\alpha_1,\alpha_2,\alpha_3)$ and $\bm c=(1,1,1,1)^T$. Defining $B=\begin{pmatrix}A\\\bm c^T\end{pmatrix}$, the conditions become
\[
B\bm u=\binom{\bm b}{\ell} \in \Z^3\times\{\ell\},
\]
where $\bm b:=A\bm u\in[0,2)^3$. Using the inverse matrix
\[
B^{-1}=\begin{pmatrix}-1/2 & -1/2 & -1/2 & 1 \\ 1/2 & 1/2 & -1/2 & 0 \\ 1/2 & -1/2 & 1/2 & 0 \\ -1/2 & 1/2 & 1/2 & 0\end{pmatrix},
\]
we compute $\bm u=B^{-1}\binom{\bm b}{\ell}$ for each integer vector $\bm b\in[0,2)^3$ and determine when $\bm u\in[0,1)^4$:
\begin{itemize}
\item $\bm b=(0,0,0)^T$: $\bm u=(\ell,0,0,0)^T\in[0,1)^4 \Rightarrow \ell\in[0,1)$. 
\item $\bm b=(1,0,0)^T, (0,1,0)^T, (0,0,1)^T$: A coordinate of $\bm u$ is $-1/2$, so no solution.
\item $\bm b=(1,1,0)^T, (1,0,1)^T, (0,1,1)^T$: A coordinate of $\bm u$ is $1$, so no solution.
\item $\bm b=(1,1,1)^T$: $\bm u=(-3/2+\ell,1/2,1/2,1/2)^T\in[0,1)^4 \Rightarrow \ell\in[3/2,5/2)$. 
\end{itemize}

Therefore, $\#D(\sigma,\ell)=\mathbf{1}_{[0,1)}(\ell)+\mathbf{1}_{[3/2,5/2)}(\ell)$, and
\[
L(\sigma,t)=\begin{cases}
B_3(t)+B_1(t) & 0\leq\{t\}<1/2,\\
B_3(t)+B_2(t) & 1/2\leq\{t\}<1.
\end{cases}
\]
More explicitly, write $L(\sigma,t)$ as a quasi-polynomial in $t$ on each interval:
\[
\begin{cases}
\frac{1}{3}t^3 + (1 - \{t\}) t^2 + \left(\{t\}^2 - 2\{t\} + \frac{5}{3}\right) t + \left(   - \frac{1}{3}\{t\}^3+ \{t\}^2- \frac{5}{3}\{t\}+1\right),&0\leq\{t\}<1/2,\\
\frac{1}{3}t^3 + \left(\frac{3}{2} - \{t\}\right)t^2 + \left(\{t\}^2 - 3\{t\} + \frac{13}{6}\right)t + \left( - \frac{1}{3}\{t\}^3+ \frac{3}{2}\{t\}^2- \frac{13}{6}\{t\} +1\right),&1/2\leq\{t\}<1.
\end{cases}
\]
For the interior $\mathring{\sigma}$, by proposition \ref{binomial formula for real dilations}, $L(\mathring{\sigma},t)$ is given by the sum
\[
L(\mathring{\sigma},t)=\sum_{j=1}^{4}\#\bar{D}(\sigma,\langle t\rangle+j)B'_{3-j}(t).
\]
Analogously, we need $\bm u\in(0,1]^4$, which gives $\bm b\in(0,2]^3$. The analysis is parallel:
\begin{itemize}
\item $\bm b=(1,1,1)^T$: $\ell\in(3/2,5/2]$. $\bm b=(2,2,2)^T$: $\ell\in(3,4]$.
\item All other $\bm b\in(0,2]^3\cap\Z^3$: No solution.
\end{itemize}
Therefore, $\#\bar{D}(\sigma,\ell)=1_{(\frac{3}{2},\frac{5}{2}]}+1_{(3,4]}$, and
\[
L(\mathring{\sigma},t)=\begin{cases}
B_{-1}'(t)+B_0'(t) & -1<\langle t\rangle\leq -1/2,\\
B_{-1}'(t)+B_1'(t) & -1/2<\langle t\rangle\leq 0.
\end{cases}
\]
More explicitly, write $L(\mathring{\sigma},t)$ as a quasi-polynomial in $t$ on each interval:  
\[
\begin{cases}
\frac{1}{3}t^3 - \left(\langle t\rangle + \frac{3}{2}\right)t^2 + \left(\langle t\rangle^2 + 3\langle t\rangle + \frac{13}{6}\right)t - \left(\frac{1}{3}\langle t\rangle^3 + \frac{3}{2}\langle t\rangle^2 + \frac{13}{6}\langle t\rangle + 1\right),&-1<\langle t\rangle\leq -1/2,\\
\frac{1}{3}t^3 - \bigl(\langle t\rangle + 1\bigr)t^2 + \left(\langle t\rangle^2 + 2\langle t\rangle + \frac{5}{3}\right)t - \left(\frac{1}{3}\langle t\rangle^3 + \langle t\rangle^2 + \frac{5}{3}\langle t\rangle + 1\right),&-1/2<\langle t\rangle\leq 0.
\end{cases}
\]

The reciprocity law $L(\mathring{\sigma},-t)=(-1)^{\dim\sigma}L(\sigma,t)$ holds for this example. Since $\langle -t\rangle=-\{t\}$, one can verify that the values of $L(\mathring{\sigma},-t)$ and $-L(\sigma,t)$ match on each interval.

\bibliography{Cao-Chen.bib}

\end{document}